\documentclass[12pt,reqno,a4paper]{amsart}
\usepackage{setspace}
\usepackage{fullpage}
\usepackage{datetime}

\usepackage{latexsym} 

\usepackage[utf8]{inputenc}
\usepackage[english]{babel}

\usepackage{amsfonts,amsmath,amssymb,amsbsy,amsthm,enumerate}

\usepackage[mathscr]{euscript}
\usepackage{pstricks}
\usepackage{multirow}
\usepackage{verbatim}
\usepackage{color}
\usepackage{comment} 
\usepackage[normalem]{ulem} 
\usepackage{bm}
\usepackage{enumitem}
\usepackage{subcaption} 

\definecolor{roxo}{HTML}{8613B0}

\newtheorem{theorem}{Theorem}[section]

\newtheorem{conjecture}[theorem]{Conjecture}

\newtheorem{lemma}[theorem]{Lemma}
\newtheorem{corollary}[theorem]{Corollary}

\newtheorem{claim}[theorem]{Claim}

\newtheorem{definition}[theorem]{Definition}

\newcommand{\aset}[1]{\left\{#1\right\}}

\usepackage{tikz}
\usepackage{xifthen}	
\usepackage{subcaption}
\usetikzlibrary{calc,positioning,decorations.pathmorphing,decorations.pathreplacing}
\usepackage{./Figures/figstyle}

\def\start{\mathop{\text{\rm start}}\nolimits}
\def\fim{\mathop{\text{\rm end}}\nolimits}

\newcommand\D{\mathcal{D}}

\usepackage{graphicx}

\title[Decomposing highly connected graphs into paths of length five]
{Decomposing highly connected graphs \\ into paths of length five}

\author [ ]{F. Botler, ~ G. O. Mota, ~  M. T. I. Oshiro, ~ Y. Wakabayashi}

\dedicatory{\small {\rm Instituto de Matem\'atica e Estat\'{\i}stica \\ Universidade de S\~ao 
Paulo, Brazil }} 

\thanks{This research has been partially supported by CNPq
  Projects (Proc. 477203/2012-4 and {456792/2014-7}), Fapesp Project
  (Proc. 2013/03447-6) and MaCLinC Project of Numec/USP, Brazil.
  F. Botler is supported by Fapesp (Proc. 2014/01460-8 and
  2011/08033-0), G. O.  Mota is supported by Fapesp
  (Proc. 2013/11431-2 and 2013/20733-2), M. T. I. Oshiro is supported
  by Capes, and Y. Wakabayashi is partially supported by CNPq Grant
  (Proc. 303987/2010-3).\\
 Email:{\texttt{{\{fbotler|mota|oshiro|yw\}@ime.usp.br}}}}

\date{\today, \currenttime}

\begin{document}

\begin{abstract}
  Barát and Thomassen (2006) posed the following decomposition
  conjecture: for each tree \(T\), there exists a natural number
  \(k_T\) such that, if \(G\) is a \(k_T\)-edge-connected graph and
  \(|E(G)|\) is divisible by \(|E(T)|\), then \(G\) admits a
  decomposition into copies of \(T\).
  In a series of papers, Thomassen verified this conjecture for stars,
  some bistars, paths of length \(3\), and paths whose length is a
  power of $2$.  We verify this conjecture for paths of
  length \(5\).
\end{abstract}

\maketitle
\onehalfspacing

\section{Introduction}

A \emph{decomposition} \(\D\) of a graph \(G\) is a set
\(\{H_1,\ldots,H_k\}\) of pairwise edge-disjoint subgraphs of \(G\)
whose union is \(G\).  If each subgraph \(H_i\), \(1\leq i\leq k\), is
isomorphic to a given graph \(H\), then we say that \(\D\) is an
\emph{\(H\)-decomposition} of \(G\).  

A well-known result of Kotzig (see~\cite{Bondy90,Kotzig57}) states
that a connected graph \(G\) admits a decomposition into paths of
length~\(2\) if and only if \(G\) has an even number of edges.  Dor
and Tarsi~\cite{DorTarsi97} proved that the problem of deciding
whether a graph has an \(H\)-decomposition is NP-complete whenever
\(H\) is a connected graph with at least \(3\)~edges.  It is then
natural to consider special classes of graphs~\(H\), and look for
sufficient conditions for a graph~\(G\) to admit an \(H\)-decomposition.
One class of graphs that has been studied from this point of view is
that of paths, in special when the input graph $G$ is regular. 
A pioneering work on this topic dates back to 1957, and
although some others have followed, a number of questions 
remain open~\cite{BoFo83,FaGeKo10,HeLiYu99,Kotzig57}.
For the special case in which $H$ is a tree, Barát and
Thomassen~\cite{BaTh06} proposed the following conjecture.

\begin{conjecture}\label{conj:dec}
For each tree $T$, there exists a natural number $k_T$ such that, if
  $G$ is a \hbox{$k_T$-edge-connected} graph and $|E(G)|$ is divisible by
  $|E(T)|$, then \(G\) admits a \hbox{\(T\)-decomposition.}

\end{conjecture}

  Barát and Thomassen~\cite{BaTh06} proved that
  Conjecture~\ref{conj:dec} in the special case $T$ is the claw
  \(K_{1,3}\) is equivalent to Tutte's weak \(3\)-flow
  conjecture, posed by Jaeger~\cite{Jaeger88}. They also 
  observed that this conjecture is false if, instead of a tree, we
  consider a graph that contains a cycle.

 Since 2008 many results on this conjecture have been found by
  Thomassen~\cite{Th08b,Th08a,Th12,Th13a,Th13b}. He has verified that
  this conjecture holds for paths of length~\(3\), stars, a family of
  bistars, and paths whose length is a power of~$2$. Recently, we
  learned that Merker~\cite{Me15+} proved that
  Conjecture~\ref{conj:dec} holds for trees with diameter at most~$4$
  and also for some trees with diameter at most~$5$, including $P_5$,
  the path of length five.

In this paper we will focus on the following version of Conjecture~\ref{conj:dec} for bipartite graphs.

\begin{conjecture}\label{conj:dec-bip}
For each tree $T$, there exists a natural number
$k'_T$ such that, if $G$ is a $k'_T$-edge-connected bipartite graph
and $|E(G)|$ is divisible by $|E(T)|$, then \(G\) admits a \(T\)-decomposition.
\end{conjecture}

Recently, Barát and Gerbner, and Thomassen independently proved that
Conjectures~\ref{conj:dec}~and~\ref{conj:dec-bip}  are equivalent. The next theorem states this result precisely.

\begin{theorem}[Barát--Gerbner~\cite{BaGe14}; Thomassen~\cite{Th13a}]\label{theorem:conj-equivalence}
Let $T$ be a tree on $t$ vertices, with $t>4$. The following two statements are
equivalent.
\begin{itemize}
\item[(i)] There exists a natural number $k'_T$ such that, if $G$ is a
    $k'_T$-edge-connected bipartite graph and $|E(G)|$ divisible
    by $|E(T)|$, then $G$ admits a $T$-decomposition.
\item[(ii)] There exists a natural number $k_T$ such that, if $G$ is a
    $k_T$-edge-connected graph and $|E(G)|$ is divisible by
    $|E(T)|$, then $G$ admits a $T$-decomposition.
\end{itemize}
	Furthermore, 
	\(
		k_T\leq 4k'_T + 16(t-1)^{6t-5}
	\)
	and, if in addition $T$ has diameter at most \(3\), then
	\(
		k_T\leq 4k'_T + 16t(t-1).
	\)
\end{theorem}

In this paper we verify Conjecture~\ref{conj:dec-bip} (and
Conjecture~\ref{conj:dec}) in the special case $T$ is the path of
length five.  More specifically, we prove that \(k'_{P_5} \leq 48\).

In our proof we use a generalization of the technique used by
  Thomassen~\cite{Th08b} to obtain an initial decomposition into
  trails of length~\(5\). Then, inspired by the ideas used
  in~\cite{BoMoWa14+}, we obtain a result that allows us to
  ``disentangle'' the undesired trails of this initial decomposition
  and construct a pure path decomposition.
  
  The paper is organized as follows. In Section~\ref{sec:def} we give
  some definitions, establish the notation and state some auxiliary
  results needed in the proof of our main result, presented in
  Section~4.  In Section~\ref{sec:canonical-decomposition} we prove
  that a highly edge-connected graph admits a ``canonical''
  decomposition into paths and trails of length~$5$ satisfying certain
  properties.  In Section~\ref{sec:main} we show how to switch edges
  between the elements of the above decomposition and obtain a
  decomposition into paths of length~$5$. We finish with some
  concluding remarks in Section~$5$.


An extended abstract~\cite{BoMoOsWa15-LAGOS} of this work was
    presented at the conference \textsc{lagos 2015}. Further
    improvements were obtained since then, and these are incorporated
    into this work. In special, a bound for $k'_{P_5}$ was improved
    from $134$ to $48$.  Moreover, we~\cite{BoMoOsWa15+thomassen} have
    been able to generalize some of the ideas presented here to prove
    that Conjecture~\ref{conj:dec} holds for paths of any given
    length. We consider that the ideas and techniques presented in
    this paper are easier to be understood, and they can be seen as a
    first step towards obtaining more general results not only for
    paths of fixed length, but also for other type of
    results~\cite{BoMoOsWa15+reg}. As the generalization is not so
    straightforward, we believe that those interested on the more
    general case will benefit reading this work first.


\section{Notation and auxiliary results}\label{sec:def}

The basic terminology and notation used in this paper are standard
(see, e.g.~\cite{Bo98,Di10}).  
A \emph{graph} has no loops or multiple edges. A \emph{multigraph} may
have multiple edges but no loops.  A \emph{directed graph} (resp. \emph{directed
multigraph}) is a graph (resp. multigraph) together with an orientation
of its edges.  More precisely, a directed graph (resp. multigraph) is
a pair $\vec G=(V,A)$ consisting of a vertex-set $V$ and a set $A$ of
ordered pairs of distinct vertices, called \emph{directed edges} (or,
simply, \emph{edges}).
%
When a pair $(V,A)$ that defines a (directed) graph $G$ is not
  given explicitly, such a pair is assumed to be $(V(G),A(G))$.  Given
  a directed graph $\vec G$, the set of edges obtained by removing
  the orientation of the directed edges in $A(\vec G)$ is denoted by
  $\hat A(\vec G)$ and is called the \emph{underlying edge-set} of
  $A(\vec G)$.  We denote by $G$ the \emph{underlying graph} of $\vec
  G$, that is, the graph with vertex-set $V(\vec G)$ and edge-set
  $\hat A(\vec G)$.  We say that \(\vec G\) is \(k\)-edge-connected if
  \(G\) is \(k\)-edge-connected. We denote by $G=(A\cup B,E)$ a
  bipartite graph $G$ on vertex classes $A$ and $B$.

We denote by $Q=v_0v_1\cdots v_k$ a sequence of vertices of a graph
$G$ such that $v_iv_{i+1}\in E(G)$, for
$i=0,\ldots,k-1$. If the edges $v_iv_{i+1}$,
$i=0,\ldots,k-1$, are all disctint, then we say that $Q$ is a
\emph{trail}; and if all vertices in $Q$ are distinct, then we say
that $Q$ is a \emph{path}. The \emph{length} of $Q$ is $k$ (the number of its edges).  
A path  of length~$k$ is denoted by $P_k$, and is also called a \emph{$k$-path}.
If \(\vec Q= v_0v_1\cdots v_k\) is a sequence of vertices of a directed graph \(\vec G\),
we say that \(\vec Q\) is a \emph{path} (resp. \emph{trail}) if \(Q\) is 
a path (resp. trail) in \(G\).

We say that a directed graph \(\vec H\) is a \emph{copy} of a graph
\(G\) if \(H\) is isomorphic to~\(G\).
We say that a set $\{H_1,\ldots,H_k\}$ of graphs is a
\emph{decomposition} of a graph $G$ if $\bigcup_{i=1}^k
E(H_i) = E(G)$ and $E(H_i)\cap E(H_j)=\emptyset$ for all $1\leq i<j\leq
k$. 
For a directed graph, the definition is analogous.
Let~\(\mathcal{H}\) be a family of graphs. An
\emph{\(\mathcal{H}\)-decomposition} $\D$ of~\(\vec G\) is a
decomposition of~\(\vec G\) such that each element of $\D$ is
a copy of an element of~\(\mathcal{H}\). If~\(\mathcal{H} =
\{H\}\) we say that $\D$ is an~\emph{\(H\)-decomposition}.

In what follows, we present some concepts and auxiliary results
  that will be used in the forthcoming sections. We assume here that
  the set of natural numbers does not contain zero.

\subsection{Vertex splittings}
Let $G=(V,E)$ be a graph and $x$ a vertex of \(G\). 
A set $S_x=\{d_1,\ldots,d_{s_x}\}$ of $s_x$ natural numbers is called a 
\emph{subdegree sequence for $x$} if
$d_1+\ldots+d_{s_x}=d_G(x)$. We say that a graph $G'$ is obtained 
by an \emph{$(x,S_x)$-splitting} 
of \(G\)
if $G'$ is composed of 
$G-x$ together with $s_x$ new vertices $x_1,\ldots,x_{s_x}$ and {$d_G(x)$ new edges} satisfying the 
following conditions: 
\begin{itemize}
 \item $d_{G'}(x_i)=d_i$, for $1\leq i\leq s_x$;
 \item $\bigcup_{i=1}^{s_x} N_{G'}(x_i) = N_G(x)$.
\end{itemize}
Let $G$ be a graph and consider a set $V'=\{v_1,\ldots,v_r\}$ of $r$
vertices of $G$. Let $S_{v_1},\ldots,S_{v_r}$ be subdegree sequences
for $v_1,\ldots,v_r$, respectively. {Let $H_1,\ldots,H_r$ be
  graphs obtained as follows: $H_1$ is obtained by a
  $(v_1,S_{v_1})$-splitting of $G$, the graph $H_2$ is obtained by a
  $(v_2,S_{v_2})$-splitting of $H_1$, and so on, up to $H_r$, which is
  obtained by a $(v_r,S_{v_r})$-splitting of $H_{r-1}$.} In this case,
we say that each graph $H_i$ is a
\emph{$\{S_{v_1},\ldots,S_{v_i}\}$-detachment} of $G$. Roughly, a
detachment of a graph $G$ is a graph obtained by successive
applications of splitting operations on vertices of $G$ (see 
Figure~\ref{fig:detachment}). 

\begin{figure}[h]
	\centering
	\begin{tikzpicture}[scale = 1.5]


	\node (0) [black vertex] at (.5,2.5) {};
	
	\node (1) [black vertex] at (-.5,1.5) {};
	\node (2) [red vertex] at (1.5,1.5) {};
	
	\node (3) [black vertex] at (-.5,.5) {};
	\node (4) [red vertex] at (.5,.5) {};
	\node (5) [black vertex] at (1.5,.5) {};
	
	\node (6) [black vertex] at (-.5,-.5) {};
	\node (7) [black vertex] at (1.5,-.5) {};

	\node (a) [] at ($(0) - (0.25,0)$) {$a$};
	\node (b) [] at ($(1) - (0.25,0)$) {$b$};
	\node (c) [] at ($(2) + (0.25,0)$) {$c$};
	\node (d) [] at ($(3) - (0.25,0)$) {$d$};
	\node (e) [] at ($(4) - (0.0,+0.25)$) {$e$};
	\node (f) [] at ($(5) + (0.25,0)$) {$f$};
	\node (g) [] at ($(6) - (0.25,0)$) {$g$};
	\node (h) [] at ($(7) + (0.25,0)$) {$h$};
	
	\node (G) [] at ($(4) + (0.0,-1.5)$) {$G$};
	
	\draw[thick] (0) -- (1);
	\draw[thick] (0) -- (2);
	\draw[thick] (1) -- (2);
	\draw[thick] (1) -- (3);
	\draw[thick] (1) -- (4);
	\draw[thick] (2) -- (4);
	\draw[thick] (2) -- (5);
	\draw[thick] (3) -- (4);
	\draw[thick] (3) -- (6);
	\draw[thick] (4) -- (5);
	\draw[thick] (4) -- (6);
	\draw[thick] (4) -- (7);
	\draw[thick] (5) -- (7);
	\draw[thick] (6) -- (7);

\end{tikzpicture}\hspace{1cm}
	\begin{tikzpicture}[scale = 1.5]


	\node (0) [black vertex] at (.5,2.5) {};
	
	\node (1) [black vertex] at (-.5,1.5) {};
	
	\node (2a) [red vertex] at (1.35,1.55) {};
	\node (2b) [red vertex] at (1.5,1.4) {};
	
	\node (3) [black vertex] at (-.5,.5) {};
	
	\node (4a) [red vertex] at (.5,.67) {};
	\node (4b) [red vertex] at (.38,.45) {};
	\node (4c) [red vertex] at (.62,.45) {};
	
	\node (5) [black vertex] at (1.5,.5) {};
	
	\node (6) [black vertex] at (-.5,-.5) {};
	\node (7) [black vertex] at (1.5,-.5) {};

	\node (a) [] at ($(0) - (0.25,0)$) {$a$};
	\node (b) [] at ($(1) - (0.25,0)$) {$b$};
 	\node (c1) [] at ($(2a) + (0.22,0.1)$) {$c_1$};
 	\node (c2) [] at ($(2b) + (0.25,0)$) {$c_2$};
	\node (d) [] at ($(3) - (0.25,0)$) {$d$};
	\node (e1) [] at ($(4a) - (0.0,-0.25)$) {$e_1$};
	\node (e2) [] at ($(4b) - (0.06,+0.25)$) {$e_2$};
	\node (e3) [] at ($(4c) - (-0.06,+0.25)$) {$e_3$};
	\node (f) [] at ($(5) + (0.25,0)$) {$f$};
	\node (g) [] at ($(6) - (0.25,0)$) {$g$};
	\node (h) [] at ($(7) + (0.25,0)$) {$h$};
	
	\node (H) [] at ($(4a) + (0.0,-1.67)$) {$H$};
	
	\draw[thick] (0) -- (1);
	\draw[thick] (0) -- (2a);
	\draw[thick] (1) -- (2a);
	\draw[thick] (1) -- (3);
	\draw[thick] (1) -- (4a);
	\draw[thick] (2b) -- (4a);
	\draw[thick] (2b) -- (5);
	\draw[thick] (3) -- (4b);
	\draw[thick] (3) -- (6);
	\draw[thick] (4c) -- (5);
	\draw[thick] (4b) -- (6);
	\draw[thick] (4c) -- (7);
	\draw[thick] (5) -- (7);
	\draw[thick] (6) -- (7);

\end{tikzpicture}
        \caption{{A graph $G$ and a graph $H$ that is an
          $\{S_c,S_e\}$-detachment of~$G$, where $S_c=\{2,2\}$ and $S_e=\{2,2,2\}$. }}
	\label{fig:detachment}
\end{figure}
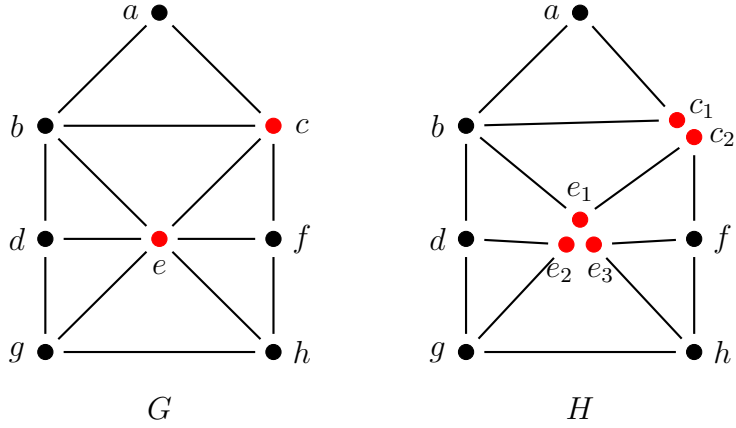

The next result provides sufficient conditions for the existence of
$2k$-edge-connected detachments of $2k$-edge-connected graphs.

\begin{lemma}[Nash--Williams~\cite{Na85}]\label{lemma:splitting-of-vertices}
  Let $k$ be a natural number, and \(G\) be a
    \(2k\)-edge-connected graph with 
    $V(G)=\{v_1,\ldots,v_n\}$. For every $v$ in $V(G)$,  let
      $S_{v}=\{d^v_1,\ldots,d^v_{s_{v}}\}$ be a {subdegree sequence for
      $v$ such that $d^v_i\geq 2k$ for $i=1,\ldots,s_{v}$. Then, there
      exists a $2k$-edge-connected} 
      $\{S_{v_1},\ldots,S_{v_{n}}\}$-detachment of $G$.
\end{lemma}

\subsection{Edge liftings}
Let $G = (V,E)$ be a graph that contains vertices \(u,v,w\) such that $uv$, $vw \in E$. The multigraph 
$G' = \big(V, (E \setminus 
\aset{uv,vw}) \cup \aset{uw}\big)$ is called a \emph{\(uw\)-lifting}
(or, simply, a \emph{lifting}) at \(v\). 
If for all distinct pairs $x$, 
$y \in V \setminus \aset{v}$, the maximum number of edge-disjoint paths between 
$x$ 
and $y$ in $G'$ is the same as in $G$, then the lifting at \(v\) is called 
\emph{admissible}.
If \(v\) is a vertex of degree \(2\),
then the lifting at \(v\) is always admissible. 
This lifting together with the deletion of \(v\) 
is called a \emph{supression} of \(v\).

The next result, known as Mader's Lifting Theorem, presents  
conditions for a multigraph to have an admissible lifting.

\begin{theorem}[Mader~\cite{Ma78}]\label{thm:Mader}
	Let $G$ be a finite multigraph and let $v$ be a vertex of $G$
        that is not a cut-vertex. If $d_G(v) \geq 4$ and  $v$ has at least 2 neighbors, 
	 then there exists an admissible lifting at $v$.
\end{theorem}

The next lemma will be useful to apply Mader's Lifting Theorem.  For
two vertices $x$,$y$ in a graph $G$, we denote by \(p_G(x,y)\) the
maximum number of edge-disjoint paths between \(x\) and \(y\) in $G$.

\begin{lemma}\label{lemma:non-cut-vertex}
   Let $k$ be a natural number. If \(G\) is a multigraph and $v$
    is a vertex in~$G$ such that $d(v)<2k$ and $p_G(x,y)\geq k$ for
    any two distinct neighbors \(x,y\) of \(v\), then $v$ is not a
    cut-vertex of~$G$.
\end{lemma}

\begin{proof}
  Let $k$, $G$ and $v$ be as in the hypothesis of the lemma. Suppose,
  by contradiction, that \(v\) is a cut-vertex. Let \(G_x\) and
  \(G_y\) be two components of \(G -v\).  Let \(x \in V(G_x)\) and
  \(y\in V(G_y)\) be two neighbors of \(v\).  By hypothesis, $G$
    has at least \(k\) edge-disjoint paths joining \(x\) to \(y\).
  Since \(v\) is a cut-vertex, each of these paths must contain \(v\).
  Thus, \(d(v) \geq 2k\), a contradiction.
\end{proof}

\subsection{Some consequences of high connectivity}

If \(G\) is a graph that contains \(2k\) pairwise edge-disjoint
spanning trees, then, clearly, \(G\) is \hbox{\(2k\)-edge-connected.}

The converse is not true, but as the following result shows,
  every \(2k\)-edge-connected graph contains \(k\) such trees.

\begin{theorem}[Nash-Williams~\cite{Na61}; Tutte~\cite{Tu61}]
\label{thm:edge-disjoint-spanning-trees}
	Let \(k\) be a natural number. 
	If \(G\) is a \(2k\)-edge-connected graph, 
	then \(G\) contains \(k\) pairwise edge-disjoint spanning trees.
\end{theorem}

We state now a result (Theorem~\ref{thm:Thomassen12}) that we shall
use in the proof of Lemma~\ref{lem:tight-good-initial-decomposition}. 
The latter allows us to treat highly edge-connected bipartite graphs
as regular bipartite graphs; it is a slight generalization of
Proposition \(2\) in~\cite{Th13a}.  Given an orientation \(O\) of a
graph \(G\), we denote by \(d^+_O(v)\) the outdegree of \(v\) in
\(O\).

\begin{theorem}[Lov{\'a}sz--Thomassen--Wu--Zhang~\cite{LoThWuZh13}]\label{thm:Thomassen12}
  Let \(k\geq 3\) be an odd natural number and \(G\)  a \((3k-3)\)-edge-connected
  graph.  Let \(p:V(G)\to\{0,\ldots,k-1\}\) be such that
  \(\sum_{v\in V(G)}p(v) \equiv |E(G)| \pmod{k}\).
  Then there is an orientation \(O\) of \(G\)
  such that \(d^+_O(v) \equiv p(v) \pmod{k}\),
  for every vertex \(v\) of \(G\).
\end{theorem}

\begin{lemma}
\label{lem:tight-good-initial-decomposition}
	Let \(k\geq 3\) and \(r\) be natural numbers, $k$ odd.
	If \(G=(A_1 \cup A_2,E)\) is a \((6k-6 + 4r)\)-edge-connected bipartite 
	graph and \(|E|\) is divisible by \(k\), 
	then \(G\) admits a decomposition into two spanning 
	\(r\)-edge-connected graphs \(G_1\) and \(G_2\) such that,
	the degree in \(G_i\) of each vertex of \(A_i\) is divisible by \(k\),
	for $i=1,2$.
\end{lemma}

\begin{proof}
  Let $k$, $r$ and \(G=(A_1 \cup A_2,E)\) be as stated in the
  lemma.  By Theorem~\ref{thm:edge-disjoint-spanning-trees}, \(G\)
  contains \(3k-3 + 2r\) pairwise edge-disjoint spanning trees.  Let
  \(H_1\) be the union of \(r\) of these trees, let \(H_2\) be the
  union of other \(r\) of these trees, and let \(H_3 = G-E(H_1)
  -E(H_2)\).  Thus, \(H_1\) and \(H_2\) are \(r\)-edge-connected, and
  \(H_3\) is \((3k-3)\)-edge-connected.

  Take \(p\colon V(H_3)\to\{0,\ldots,k-1\}\) such that \(p(v) \equiv
  (k-1)d_{H_1}(v)\pmod{k}\) if \(v\) is a vertex of \(A_1\), and \(p(v) \equiv
  (k-1)d_{H_2}(v)\pmod{k}\)  if \(v\) is a vertex of \(A_2\).
  Thus, the following holds, where the congruences are taken modulo~$k$.
	\begin{align*}
	\sum_{v \in V(G)} p(v) 
		& \quad = \quad \sum_{v\in A_1} p(v) +  \sum_{v\in A_2} p(v)\\
		& \quad \equiv  \quad (k-1)(|E(H_1)| + |E(H_2)|)\\
		& \quad \equiv  \quad (k-1)(|E| - |E(H_3)|)\\
		& \quad \equiv  \quad { k\,(|E|-|E(H_3)|) -|E| + |E(H_3)| } \\
		& \quad \equiv \quad |E(H_3)|.
	\end{align*}
	
	Since \(H_3\) is a \((3k-3)\)-edge-connected {spanning
        subgraph of~$G$,} by Theorem~\ref{thm:Thomassen12} there is an
        orientation \(O\) of \(H_3\) such
        that \(d^+_O(v)\equiv p(v) \pmod{k}\) for every $v\in V(H_3)=V(G)$.
	{For $i=1,2$, let $G_i$ be the graph \(H_i\) together with the
        edges of \(H_3\) that leave \(A_i\) in the orientation $O$
        (note that, $E=E(G_1)\cup E(G_2)$).  Thus, \(d_{G_i}(v) =
        d_{H_i}(v) + d^+_{O}(v) \equiv k\, d_{H_i}(v) \equiv 0 \pmod{k}\)
        for every vertex \(v\) in \(A_i\), and moreover, $G_i$ is
        \(r\)-edge-connected (because it contains $H_i$).}

\end{proof}

We note that in
  Lemma~\ref{lem:tight-good-initial-decomposition} we have $k$ odd and
  the $(6k-6+4r)$-edge-connectivity of $G$ is a consequence of the
  $(3k-3)$-edge-connectivity in the statement of
  Theorem~\ref{thm:Thomassen12}.  When $k$ is even, we can also prove
  an analogous result, changing the edge-connectivity of $G$ to
  $6k-4+4r$. For that, we only have to use a slightly weaker form of
  Theorem~\ref{thm:Thomassen12} for $k$ even, according to which, as 
  stated in~\cite{LoThWuZh13}, one may change the bound
  $(3k-3)$ to $(3k-2)$.

Given a graph \(G\) and a natural number \(r\),
an \emph{\(r\)-factor} in \(G\) is an \(r\)-regular spanning subgraph of \(G\). 
The following two results on $r$-factors in regular multigraphs
  will be used later.

\begin{theorem}[Von Baebler~\cite{Ba37} (see also~{\cite[Theorem 2.37]{AkKa11}})]\label{thm:Kano}
Let $r \geq 2$ be a natural number, and $G$ be an $(r-1)$-edge-connected 
$r$-regular multigraph of even order. Then $G$ has a $1$-factor. 
\end{theorem}

\begin{theorem}[Petersen~\cite{Pe1891}]\label{thm:Petersen}
	If \(G\) is a \(2k\)-regular multigraph,
	then \(G\) admits a decomposition into \(2\)-factors.
\end{theorem}

\section{Fractional factorizations and canonical decompositions}\label{sec:canonical-decomposition}

In this section we prove that every $4$-edge-connected bipartite graph
$G=(A\cup B,E)$ such that the degree of each vertex in $A$ is divisible
by~$5$ admits a special decomposition, which we call ``fractional
factorization'' (see Subsection~\ref{sub:frac}).  Moreover, if \(G\)
is \(6\)-edge-connected, then such a factorization guarantees
  that we can construct a decomposition of \(G\) into trails of length
  \(5\) with some special properties (see Subsection~\ref{sub:can}).

\subsection{Fractional factorizations}~\label{sub:frac}

To simplify notation, if $F$ is a set of edges of a graph~$G$,
we write $d_F(v)$ to denote the degree of $v$ in $G[F]$,
the subgraph of \(G\) induced by \(F\).
If \(F\) is a set of edges of a directed graph \(\vec G\), we write \(d^+_F(v)\) 
(resp. \(d^-_F(v)\))
to denote the outdegree (resp. indegree) of \(v\) in~\(\vec{G}[F]\). 

\begin{definition}\label{def:special-decomposition}
	Let $\vec G$ be a bipartite directed graph with vertex classes \(A\) and \(B\),
	and  such that {the degree
	of each vertex in $A$ is divisible by~$5$.} We say that $\vec G$
	admits a \emph{fractional factorization} $(M,F,H)$ for \(A\) if $A(\vec G)$
	can be decomposed into three edge-sets $M$, $F$ and $H$ such
	that the following holds.
\begin{enumerate}[label=(\roman*)]
 \item Every edge in $M$ is directed from $B$ to $A$;
 \item For every $v\in A$, we have $d^-_{F}(v)=d^+_{F}(v)=d^-_{H}(v)=d^+_{H}(v)=d^-_{M}(v)=d(v)/5$;
 \item For every $v\in B$, we have $d^-_{F}(v)=d^+_{F}(v)$ and $d^-_{H}(v)=d^+_{H}(v)$.
\end{enumerate}
\end{definition}

\smallskip

\begin{lemma}\label{lemma:special-decomposition}
{Let $G=(A\cup B,E)$ be a $4$-edge-connected bipartite graph such
    that the degree of each vertex in $A$ is divisible by~$5$.}  Then,
    $G$ is the underlying graph of a directed graph $\vec G$ that
    admits a fractional factorization $(M,F,H)$ for \(A\).
\end{lemma}

\begin{proof}
{Let \(G=(A\cup B,E)\) be as stated in the hypothesis of the lemma.}
First, we want to apply Lemma~\ref{lemma:splitting-of-vertices} to $G$
and obtain a \(4\)-edge-connected graph $G'$ with maximum degree~\(7\).
To do this, for every vertex $v \in B$, we take integers $s_v \geq 1$
and $0 \leq r_v < 4$ such that $d(v) = 4s_v + r_v$.  We put
$d^v_1 = 4+r_v$ and \(d^v_2 = \cdots = d^v_{s_v} = 4\).
Furthermore, for every vertex $v\in A$, we put $s_v=d(v)/(5)$ and
$d^v_i=5$ for $1\leq i\leq s_v$.  By
Lemma~\ref{lemma:splitting-of-vertices}, applied with parameters $k=2$
and the integers $s_v$, $d^v_i$ ($1\leq i\leq s_v$) for every
$v\in V(G)$, there exists a $4$-edge-connected bipartite graph $G'$
obtained from \(G\)
by splitting each vertex \(v\)
of \(A\)
into \(s_v\)
vertices of degree \(5\),
and each vertex \(v\)
of \(B\)
into a vertex of degree $4+r_v < 8$ and $s_v-1$ vertices of degree
$4$.  Let \(A'\)
and \(B'\)
be the set of vertices of \(G'\)
obtained from the vertices of \(A\)
and \(B\),
respectively.  For ease of notation, if
\(v \in (A' \cup B') \setminus (A \cup B)\)
we also denote by \(v\)
the original vertex in \((A \cup B)\) at which we applied splitting.

The next step is to obtain a $5$-regular multigraph $G^*$ from $G'$ by 
using lifting operations. For this, we 
will add some edges to $A'$ and remove the even-degree vertices of $B'$ by successive
applications of Mader's Lifting Theorem as 
follows. 

Let \(G'_0,G'_1,\ldots, G'_\lambda\) be a maximal sequence of graphs such that \(G'_0 = G'\)
and (for \(i\geq 0\)) \(G'_{i+1}\) is the graph obtained from \(G'_i\) 
by the application of an admissible lifting
at an arbitrary vertex \(v\) of degree in \(\{4,6,7\}\).

Recall that given any two vertices of \(G'\), say \(x\) and \(y\), we denote by \(p_{G'}(x,y)\)
the maximum number of pairwise edge-disjoint paths joining \(x\) and \(y\) in \({G'}\).
We claim that \(p_{G'_i}(x,y) \geq 4\) for any \(x,y\) in \(A'\) and every \(i\geq 0\).
Clearly, \(p_{G'_0}(x,y) \geq 4\) holds for any \(x,y\) in \(A'\), since \(G'\) is \(4\)-edge-connected.
Fix \(i\geq 0\) and suppose \(p_{G'_i}(x,y) \geq 4\) holds for any \(x,y\) in \(A'\).
Let \(x,y\) be two vertices in \(A'\).
Since \(G'_{i+1}\) is a graph obtained from \(G'_i\) by the application of an admissible lifting 
at a vertex \(v\) in \(B'\), we have \(p_{G'_{i+1}}(x,y) \geq p_{G'_i}(x,y) \geq 4\).


We claim that if \(v\) is a vertex in \(B'\),
then \(d_{G'_\lambda}(v) \in \{2,5\}\).
Suppose by contradiction that there is a vertex \(v\) in \(B'\) such that
\(d_{G'_\lambda}(v) \notin \{2,5\}\).
Note that \(d_{G'_i}(u) \geq d_{G'_{i+1}}(u)\geq 2\) for every vertex \(u\) of \(G\)
and every \(0\leq i \leq \lambda\).
Since \(d_{G'}(u) \leq 7\) for every vertex \(u\) in \(V'\),
we have \(2 \leq d_{G'_i}(u) \leq 7\) for every \(0\leq i \leq \lambda\).
Therefore, \(d_{G'_i}(v) \in \{4,6,7\}\).
Since \(d_{G'_\lambda}(v) \leq 7\) and for any two neighbors \(x,y\) of \(v\) we have 
\(p_{G'_\lambda}(x,y) \geq 4\),
Lemma~\ref{lemma:non-cut-vertex} implies that \(v\) is not a cut-vertex of \(G'_\lambda\).
Then, by Mader's Lifting Theorem (Theorem~\ref{thm:Mader}) in \(G'_\lambda\),
there is an admissible lifting at \(v\).
Therefore, \(G'_0,G'_1,\ldots, G'_\lambda\) is not maximal, a contradiction.

In \(G_\lambda'\)
we may have some vertices in \(B'\)
that have degree $2$.  For every such vertex $v$, if $u$ and $w$ are
the neighbors of $v$, we apply a \(uw\)-lifting
at $v$, and remove the vertex $v$, i.e., we perform a supression of
\(v\).
Let \(G^*\)
be the graph obtained by this process.  Note that the number of
pairwise edge-disjoint paths joining two distinct vertices of \(A'\)
remains the same, and thus,
  \(p_{G^*}(x,y) = p_{G'_\lambda}(x,y)\geq 4\)
for every \(x,y\)
in \(A'\).
Furthermore, the set of vertices of $G^*$ that belong to $B'$ is an
independent set; we denote it by $B^*$ (eventually,
\(B^* = \emptyset\)).
Note that, if \(B^*\)
is nonempty, every vertex of \(B^*\)
has degree \(5\).

\begin{claim}
\(G^*\) is \(4\)-edge-connected.
\end{claim}
\begin{proof}
Let \(Y \subseteq V(G^*)\).
Suppose there is at least one vertex \(x\) of \(A'\) in \(Y\)
and at least one vertex \(y\) of \(A'\) in \(V(G^*)-Y\). 
Since there are at least \(4\) edge-disjoint paths joining \(x\) to \(y\),
there are at least \(4\) edges, each one with vertices in both \(Y\) and 
\(V(G^*)-Y\).
Now, suppose that \(A'\subset Y\) (otherwise \(A'\subset V(G^*)-Y\) and
we take \(V(G^*)-Y\) instead of \(Y\)), and then \(V(G^*)-Y \subseteq B^*\).
Since \(B^*\) is an independent set, all edges with a vertex in \(V(G^*)-Y\) 
must have the other vertex in \(A'\).
Since every vertex in \(B^*\) has degree \(5\), 
there are at least \(5\) edges, each one with vertices in both \(Y\) and \(V(G^*)-Y\). 
\end{proof}

We conclude that $G^*$ is a $4$-edge-connected $5$-regular multigraph
with vertex-set $A'\cup B^*$, where $B^*$ is an independent set.

Now we work on the multigraph $G^*$.  Since \(G^*\) is $5$-regular,
\(G^*\) has even order.  Thus, by Theorem~\ref{thm:Kano}, \(G^*\)
contains a perfect matching \(M^*\).  The multigraph \(J^* = G^* -
M^*\) is a \(4\)-regular multigraph.  By Theorem~\ref{thm:Petersen},
\(J^*\) admits a decomposition into \(2\)-factors with edge-sets, say
\(F^*\) and \(H^*\).  Thus, \(M^*\), \(F^*\), and \(H^*\) define a
partition of \(E(G^*)\).

Now let us go back to the bipartite graph $G$. 
Let \(xy\) be an edge of \(G^*\).
If \(xy\) joins a vertex of \(A'\) to a vertex of \(B^*\),
then \(xy\) corresponds to an edge of \(G\).
If \(xy\) joins two vertices of \(A'\),
then there is a vertex \(v_{xy}\) of \(B'\) and two edges of \(G'\) incident to 
it, 
\(xv_{xy}\) and \(v_{xy}y\),
such that \(xy\) was obtained by an \(xy\)-lifting 
at \(v_{xy}\) (either by an application of Mader's Lifting 
Theorem or by the supression of vertices of degree $2$).
Thus, each edge of \(G^*\) represents an edge of \(G\) or a \(2\)-path
in \(G\) such that the internal vertices of these \(2\)-paths are 
always 
in \(B\).
For every edge $xy\in E(G^*)$, define \(f(xy) = \{xy\}\) if \(xy\) joins a 
vertex of 
\(A'\) to a vertex of \(B^*\),
and \(f(xy) = \{xv_{xy},v_{xy}y\}\) if \(xy\) joins two vertices 
of \(A'\). Note that, for every edge $xy$ of $G^*$, we have $f(xy)\subset E(G)$.
For a set \(S\) of edges of \(G^*\), put \(f(S) = \cup_{e\in S} f(e)\).
The partition of \(E(G^*)\) into \(M^*\), \(F^*\) and \(H^*\) induces a 
partition of \(E(G)\)
into \(M=f(M^*)\), \(F = f(F^*)\) and \(H = f(H^*)\).

Now we construct an Eulerian orientation of \(G[F]\) and \(G[H]\)
induced by any Eulerian orientation of \(G^*[F^*]\) and \(G^*[H^*]\).
Let \(xy\) be an edge of \(G^*-M^*\) oriented from \(x\) to \(y\). 
If \(xy\) joins a vertex of \(A'\) to a vertex of \(B'\), let \(xy\) be oriented from \(x\) to \(y\)
in \(G-M\).
Otherwise, recall that \(f(xy) = \{xv_{xy},v_{xy}y\}\), 
and let \(xv_{xy}\) be oriented from \(x\) to \(v_{xy}\) in \(G-M\), and \(v_{xy}y\) be oriented
from \(v_{xy}\) to \(y\) in \(G-M\).
The obtained orientation of $G-M$ is Eulerian. Finally, orient 
all edges of $M$ from \(B\) to \(A\).
Let \(\vec G\) be the directed graph obtained by such an orientation of \(G\).

Let us prove that \((M,F,H)\) is a fractional factorization of \(\vec G\) for \(A\).
Let \(v\) be a vertex of \(A\) in \(G\) of degree \(5d'(v)\).
The vertex \(v\) is represented by \(d'(v)\) vertices in \(G^*\).
Since \(M^*\) is a perfect matching in \(G^*\), there are \(d'(v)\) edges
of \(M\) entering \(v\).
Since \(G^*[F^*]\) (resp. \(G^*[H^*]\)) is a \(2\)-factor in \(G^*\),
there are \(d'(v)\) edges of \(F\) (resp. \(H\)) entering \(v\)
and \(d'(v)\) edges of \(F\) (resp. \(H\)) leaving \(v\).
Since \(G^*[F^*]\) (resp. \(G^*[H^*]\)) is a \(2\)-factor in \(G^*\),
we have \(d_F^+(v) = d_F^-(v) = d_H^+(v) = d_H^-(v)\), concluding the proof.

\end{proof}

\subsection{Canonical decompositions}~\label{sub:can}

In this subsection we show that if a $6$-edge-connected bipartite directed graph admits
  a fractional factorization, then it admits a very special trail
  decomposition. We make precise what are the properties of such a
  special trail decompositon.

  Let $\vec G$ be a directed graph such that $A(\vec G)$ is {the union
    of pairwise} disjoint sets of directed edges $M$, $F$ and $H$.
  The following definitions refer to the triple \(\mathcal{F} =
  (M,F,H)\).  Let $T=abcde$ be a trail of length~$4$ in~$\vec G$,
  where $ab\in M$, $bc,cd\in F$ and $de\in H$.  We say that $T$ is an
  \emph{\(\mathcal{F}\)-basic path} if \(T\) is a path; and \(T\)
  is an \emph{\(\mathcal{F}\)-basic cycle} if \(T\) is a cycle (see
  Figure~\ref{fig:basics}).  Furthermore, let \(T = abcdef\) be a
  trail in $\vec G$ such that \(abcde\) is an \(\mathcal{F}\)-basic
  path.  We say that \(T\) is an \emph{\(\mathcal{F}\)-canonical path}
  if \(T\) is a path; and an \emph{\(\mathcal{F}\)-canonical trail},
  otherwise (see Figure~\ref{fig:goods}).  We say that a decomposition
  \(\mathcal{D}\) of \(\vec G\) is an \emph{\(\mathcal{F}\)-basic
    decomposition} if each element of \(\mathcal{D}\) is an
  \(\mathcal{F}\)-basic path or an \(\mathcal{F}\)-basic cycle.
  Analogously, \(\mathcal{D}\) is an \emph{\(\mathcal{F}\)-canonical
    decomposition} if each element of \(\mathcal{D}\) is an
  \(\mathcal{F}\)-canonical path or an \(\mathcal{F}\)-canonical
  trail. 


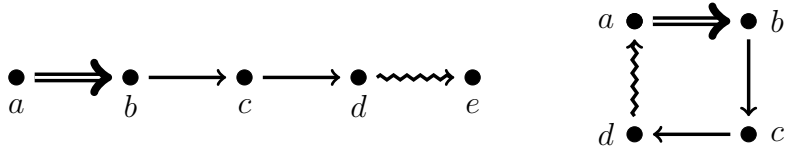
\begin{figure}[h]
\centering
	\begin{tikzpicture}[scale = 1.5]

	
	\node (-1) [	]	 at (0,-.6){};
	\node (0) [black vertex] at (-2,0) {};
	\node (1) [black vertex] at (-1,0) {};
	\node (2) [black vertex] at (0,0) {};
	\node (3) [black vertex] at (1,0) {};
	\node (4) [black vertex] at (2,0) {};

	\node (a) [] at ($(0) - (0,0.25)$) {$a$};
	\node (b) [] at ($(1) - (0,0.25)$) {$b$};
	\node (c) [] at ($(2) - (0,0.25)$) {$c$};
	\node (d) [] at ($(3) - (0,0.25)$) {$d$};
	\node (e) [] at ($(4) - (0,0.25)$) {$e$};
	
	\draw[M edge] (0) -- (1);
	\draw[thick][F1 edge] (1) -- (2);
	\draw[F1 edge] (2) -- (3);
	
	\draw[F2 edge] (3) -- (4);

\end{tikzpicture}\hspace{1cm}
	\begin{tikzpicture}[scale = 1.5]

	
	\node (0) [black vertex] at (-.5,.5) {};
	\node (1) [black vertex] at (.5,.5) {};
	\node (2) [black vertex] at (.5,-.5) {};
	\node (3) [black vertex] at (-.5,-.5) {};
	\node (4) [black vertex] at (0) {};

	\node (a) [] at ($(0) - (0.25,0)$) {$a$};
	\node (b) [] at ($(1) + (0.25,0)$) {$b$};
	\node (c) [] at ($(2) + (0.25,0)$) {$c$};
	\node (d) [] at ($(3) - (0.25,0)$) {$d$};
	
	\draw[M edge] (0) -- (1);
	\draw[F1 edge] (1) -- (2);
	\draw[F1 edge] (2) -- (3);
	\draw[F2 edge] (3) -- (4);

\end{tikzpicture}
	\caption{ An \(\mathcal{F}\)-basic path and an  \(\mathcal{F}\)-basic cycle.}
	\label{fig:basics}
\end{figure}

\begin{figure}[h]
\centering
	\begin{tikzpicture}[scale = 1.5]


	\node (-1) [	]	 at (0,-.6){};
	\node (0) [black vertex] at (-1,0) {};
	\node (1) [black vertex] at (0,0) {};
	\node (2) [black vertex] at (1,0) {};
	\node (3) [black vertex] at (2,0) {};
	\node (4) [black vertex] at (3,0) {};
	\node (5) [black vertex] at (4,0) {};
	
	\node (a) [] at ($(0) - (0,0.25)$) {$a$};
	\node (b) [] at ($(1) - (0,0.25)$) {$b$};
	\node (c) [] at ($(2) - (0,0.25)$) {$c$};
	\node (d) [] at ($(3) - (0,0.25)$) {$d$};
	\node (e) [] at ($(4) - (0,0.25)$) {$e$};
	\node (f) [] at ($(5) - (0,0.25)$) {$f$};

	\draw[M edge] (0) -- (1);
	\draw[F1 edge] (1) -- (2);
	\draw[F1 edge] (2) -- (3);
	\draw[F2 edge] (3) -- (4);
	\draw[F2 edge] (4) -- (5);

\end{tikzpicture}\hspace{1cm}
	\begin{tikzpicture}[scale = 1.5]


	\node (0) [black vertex] at (-1,1) {};
	\node (1) [black vertex] at (0,1) {};
	\node (2) [black vertex] at (1,1) {};
	\node (3) [black vertex] at (1,0) {};
	\node (4) [black vertex] at (0,0) {};

	\node (a) [] at ($(0) - (0,-0.25)$) {$a$};
	\node (b) [] at ($(1) - (0,-0.25)$) {$b$};
	\node (c) [] at ($(2) - (0,-0.25)$) {$c$};
	\node (d) [] at ($(3) - (0,0.25)$) {$d$};
	\node (e) [] at ($(4) - (0,0.25)$) {$e$};

	\draw[M edge] (0) -- (1);
	\draw[F1 edge] (1) -- (2);
	\draw[F1 edge] (2) -- (3);
	\draw[F2 edge] (3) -- (4);
	\draw[F2 edge] (4) -- (1);

\end{tikzpicture}
	\caption{An \(\mathcal{F}\)-canonical path and an \(\mathcal{F}\)-canonical trail.}
	\label{fig:goods}
\end{figure}
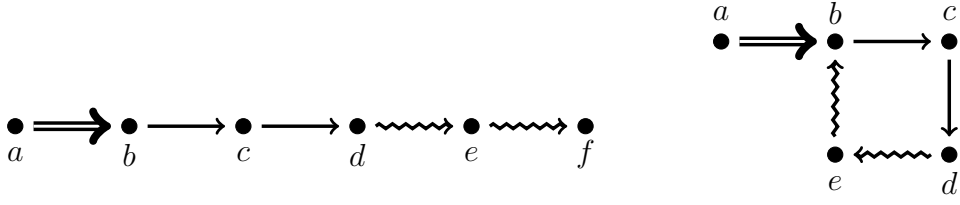

To prove the next lemma, we use some ideas inspired by the techniques in~\cite{Th08b}.

\begin{lemma}\label{lemma:decomposition-Tout}
Let $\vec G$ be a $6$-edge-connected bipartite directed graph. 
If $\vec G$ admits a fractional factorization {$\mathcal{F}$} for \(A\), then 
$\vec G$ admits an \(\mathcal{F}\)-canonical decomposition.
\end{lemma}

\begin{proof}
  Let $\vec G$ be a bipartite directed graph with vertex classes \(A\)
  and \(B\) that admits a fractional factorization \(\mathcal{F} =
  (M,F,H)\) for \(A\).  Let \(H^+(A)\) be the set of edges of \(H\)
  leaving vertices of \(A\), and let \(H^-(A)\) be the set of edges of
  \(H\) entering vertices of \(A\).
  Note that { \(\mathcal{F}' = (M,F,H^+(A))\) decomposes the edge-set}
  of \(G'=G[M \cup F \cup H^+(A)]\).

We start by proving that \(G'\) admits an \(\mathcal{F}'\)-basic path decomposition. 
For that, we first show that \(G'\) admits an \(\mathcal{F}'\)-basic decomposition 
and after we prove that there is an \(\mathcal{F}'\)-basic decomposition  without cycles.

By item (iii) of Definition~\ref{def:special-decomposition}, for every $v\in B$, we have 
$d^-_{F}(v)=d^+_{F}(v)$. 
Then, the subgraph of $G'$ induced by the edges of $F$ admits a 
$P_2$-decomposition such that the endpoints of the elements of the decomposition are in $A$. Let $\mathcal{D}_2$ be a 
$P_2$-decomposition of $G'[F]$. 
By item (ii) of Definition~\ref{def:special-decomposition}, for every 
$v\in A$, we have $d^-_{M}(v)=d^+_{F}(v)$ and $d^-_{F}(v)=d^+_{H}(v)$. 
Therefore, one can extend 
$\mathcal{D}_2$ to an \(\mathcal{F}'\)-basic decomposition of $G'$ by adding two edges to each element of 
$\mathcal{D}_2$. Precisely, for each path $xyz$ that is an element of $\mathcal{D}_2$, it is possible to 
extend it to either an \(\mathcal{F}'\)-basic path or an \(\mathcal{F}'\)-basic cycle 
by adding one edge of $M$ to $x$ and one edge of $H^+$ to $z$.

For each \(\mathcal{F}'\)-basic decomposition \(\mathcal{D}\) of \(G'\), 
let \(\rho (\mathcal{D})\)
be the number of \(\mathcal{F}'\)-basic cycles in $\mathcal{D}$. 
Let \(\mathcal{D}\) be an \(\mathcal{F}'\)-basic decomposition of \(G'\) that minimizes 
\(\rho (\mathcal{D})\) over all \(\mathcal{F}'\)-basic decompositions of \(G'\).
If \(\rho (\mathcal{D}) = 0\) then \(\mathcal{D}\) is an \(\mathcal{F}'\)-basic path decomposition of $G'$.
Thus, suppose \(\rho (\mathcal{D}) > 0\).

By definition, every element \(T\) of an \(\mathcal{F}'\)-basic decomposition 
contains exactly one 
directed path \(P\) of length two on the edges of \(F\) (see
Figure~\ref{fig:basics}), {which we call} the \emph{center} of~\(T\).
Moreover, suppose that \(P\) starts at a vertex \(x\) and ends at a
vertex \(y\).
We say that \(x\) and \(y\) are the {\emph{starting} and \emph{ending}} vertices
of \(T\), and we denote them \(\start(T)\) and \(\fim(T)\), respectively.
Note that \(x,y \in A\).

Since \(G\) is \(6\)-edge-connected
and every vertex in \(A\) has degree divisible by \(5\),
every vertex in $A$ has degree at least $10$. 
Then, since for every $v\in 
A$ we have $d^-_{F}(v)=d^+_{F}(v)=d^-_{H}(v)=d^+_{H}(v)=d^-_{M}(v)$, we conclude that every $v\in A$
contains at least two incoming edges of \(F\) and two outgoing
edges of \(F\). Therefore, given an element \(T_2\) of \(\mathcal{D}\), there exists an 
element \(T_1\) of \(\mathcal{D}\) such that 
\(\start(T_1) = \start(T_2)\) and there exists an element \(T_3\) of \(\mathcal{D}\),
such that \(\fim(T_3) = \fim(T_2)\) (note that possibly
$T_3=T_1$). Then, there is a maximal sequence  \(S = T_0,T_1,T_2,\cdots\) of elements of $\mathcal{D}$ such that \(T_0\) is an \(\mathcal{F}'\)-basic cycle and, for every $k\geq 
0$, we have \(\fim(T_{2k}) = \fim(T_{2k+1})\) and \(\start(T_{2k+1}) = \start(T_{2k+2})\) (see 
Figure~\ref{fig:sequence} for an example).

	\begin{figure}[h]
		\centering
		\begin{tikzpicture}[scale = 1.2]

	
	\node (0) [white vertex] at (1,1) {};
	\node (s0) [above] at ($(0) + (0,0.05)$) {\(s_0\)};
	\node (1) [black vertex] at (0,0) {};
	\node (t0) [below] at ($(1) - (0,0.05)$) {\(t_0\)};
	\node (2) [white vertex] at (1,-1) {};
	\node (3) [black vertex] at (2,0) {};
	\node (4) [white vertex] at (0) {};

	\node (T0) [color=green!75!blue] at ($(1) + (1,0.05)$) {\(T_0\)};
	
	\draw[color=green!75!blue][M edge] (0) -- (1);
	\draw[F1 edge][color=green!75!blue] (1) -- (2);
	\draw[F1 edge][color=green!75!blue] (2) -- (3);
	\draw[F2 edge][color=green!75!blue] (3) -- (4);
	
	\node (00) [white vertex] at (3,1) {};
	\node (s1) [above] at ($(00) + (0,0.05)$) {\(s_1\)};
	\node (11) [black vertex] at (4,0) {};
	\node (22) [white vertex] at (3,-1) {};
	\node (33) [black vertex] at (2,0) {};
	\node (t1) [below] at ($(33) - (0,0.05)$) {\(t_1\)};
	\node (44) [white vertex] at (00) {};

	\node (T1) [color=red] at ($(33) + (1,0.05)$) {\(T_1\)};

	\draw[M edge][color=red] (00) -- (11);
	\draw[F1 edge][color=red] (11) -- (22);
	\draw[F1 edge][color=red] (22) -- (33);
	\draw[F2 edge][color=red] (33) -- (44);
	
	\node (000) [white vertex] at (4,1) {};
	\node (111) [black vertex] at (4,0) {};
	\node (t2) [below] at ($(111) - (0,0.05)$) {\(t_2\)};
	\node (222) [white vertex] at (5,-1) {};
	\node (333) [black vertex] at (6,0) {};
	\node (t3) [below] at ($(333) - (0,0.05)$) {\(t_3\)};
	\node (444) [white vertex] at (6,1) {};
	\node (s2) [above] at ($(444) + (0,0.05)$) {\(s_2\)};

	\node (T2)[color=blue] at ($(111) + (1,0.05)$) {\(T_2\)};
	
	\draw[M edge][color=blue] (000) -- (111);
	\draw[F1 edge][color=blue] (111) -- (222);
	\draw[F1 edge][color=blue] (222) -- (333);
	\draw[F2 edge][color=blue] (333) -- (444);
	
	\node (x) [] at (7,0) {$\cdots$};

\end{tikzpicture}
		\caption{Example of a sequence \(T_0,T_1,T_2,\cdots\) such that \(T_0\) is an
		\hbox{\(\mathcal{F}'\)-basic cycle}  and, 
		for every $k\geq 0$, we have \(\fim(T_{2k}) = \fim(T_{2k+1})\) 
		and \(\start(T_{2k+1}) = \start(T_{2k+2})\).}
		\label{fig:sequence}
	\end{figure}
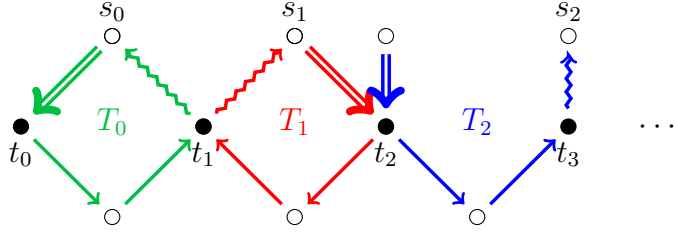
	
	Consider the sequence \(R=t_0,t_1,t_2,\cdots\) 
	of vertices of $A$ that belong to elements of~\(S\), i.e., 
	for every $k\geq 0$, we have \(t_{2k} = \start(T_{2k})\) and \(t_{2k+1} = \fim(T_{2k+1})\).
	Since \(G\) is finite, \(t_j = t_i\) for some \(0\leq i<j\). Therefore, there exists a ``cycle'' 
of elements of $\mathcal{D}$ in the sequence $S$.
Let \(i\) be the minimum integer such that \(t_i = t_j\) for some \(j > i\).
Note that if \(i\neq 0\), then \(T_{i-1} \neq T_{j-1}\).
	For each element \(T_k\) of $S$, let \(s_k\) be the vertex of $T_k$ such that either \(s_kt_{k+1} \in 
E(T_k)-F\) or \(t_{k+1}s_k \in E(T_k)-F\),
i.e, \(s_k\) is the vertex of \(T_k\) that is neighbor of \(t_{k+1}\) and is not incident to 
the edges in \(E(T_k)\cap F\).
We claim that \(s_k \neq s_0\) for some \(k > 0\). 
If \(i = 0\), then \(t_j = t_0\). Since \(T_0\) is an \(\mathcal{F}'\)-basic cycle, we have 
\(s_0t_0 \in E(T_0)-F\), from where we conclude that \(s_0t_j \notin E(T_{j-1})\), implying 
that \(s_{j-1}\neq s_0\).
Thus, suppose \(i> 0\).
Note that, since \(T_{i-1}\neq T_{j-1}\) and \(t_i = t_j\), 
we have \(s_i \neq s_j\).
Thus, at least one vertex in \(\{s_i,s_j\}\) is different from \(s_0\).

	Let \(k^*\) be the minimum integer such that \(s_{k^*} \neq s_0\). We want to disentangle the elements of 
$\mathcal{D}$ to obtain an \(\mathcal{F}'\)-basic decomposition with
fewer copies of  \(\mathcal{F}'\)-basic cycles than~$\mathcal{D}$. For that, 
consider the following notation for the elements of $\mathcal{D}$. For \(0 \leq \ell \leq k^*\), let \(T_{\ell} = 
a^\ell_{0}\,a^\ell_{1}\,a^\ell_{2}\,a^\ell_{3}\,a^\ell_{4}\) such that
$a^{\ell}_{0}a^\ell_{1}\in M$, \(a^\ell_{1}a^\ell_{2},\, 
a^\ell_{2}a^\ell_{3} \in F\) and \(a^\ell_{3}a^\ell_{4} \in H\). Thus, note that \(a^\ell_{1} = t_\ell\) and 
\(a^\ell_{3} = t_{\ell+1}\) if \(\ell\) is even, and that \(a^\ell_{1} = t_{\ell+1}\) and 
\(a^\ell_{3} = t_{\ell}\) if \(\ell\) is odd. Let
\begin{align*}
T'_0 &= a^0_{0}\,a^0_{1}\,a^0_{2}\,a^0_{3}\,\bm{a^1_{4}}; 
\\
  T'_{\ell} &= \left\{\begin{array}{l}
\bm{a^{\ell+1}_0}a^{\ell}_1\,a^{\ell}_2\,a^{\ell}_3\,\bm{a^{\ell-1}_4}
\mbox{, if \(\ell\) is odd,}\\
\bm{a^{\ell-1}_0}a^{\ell}_1\,a^{\ell}_2\,a^{\ell}_3\,\bm{a^{\ell+1}_4}\mbox{, if \(\ell\) is even,}\\
\end{array}\right.\mbox{for \(0 < \ell < k^*\);}\\
T'_{k^*} &= \left\{\begin{array}{l}
a^{k^*}_0a^{k^*}_1a^{k^*}_2a^{k^*}_3\bm{a^{k^*-1}_4}\mbox{, if \(k^*\) is odd,}\\
\bm{a^{k^*-1}_0}a^{k^*}_1a^{k^*}_2a^{k^*}_3a^{k^*}_4\mbox{, if \(k^*\) is even.}\\
\end{array}\right.
\end{align*}
	Then, \(\mathcal{D}' = \mathcal{D} - T_0 - T_1 \cdots - T_{k^*} + T'_0 + T'_1 \cdots + T'_{k^*}\) is an 
\(\mathcal{F}'\)-basic decomposition (see Figure~\ref{fig:disentangle} for an example). \begin{figure}[h]
		\centering
		\begin{tikzpicture}[scale = 1.2]

	
	\node (p1) at (1,1) {};
	\node (q1) at (3,1) {};
	\node (q2) at (5,1) {};
	\node (q3) at (7,1) {};

	\node (0) [white vertex] at (1,1) {};
	\node (a0) [left] at (0) {\small$a^0_{0},a^0_{4},a^1_{0},a^2_{4}\,\,\,$};
	\node (1) [black vertex] at (0,0) {};
	\node (a1) [left] at (1) {\small$a^0_{1}\,$};
	\node (2) [white vertex] at (1,-1) {};
	\node (a2) [below] at (2) {\small$a^0_{2}$};
	\node (3) [black vertex] at (2,0) {};
	\node (a3) [left] at (3) {\small$a^0_{3}\,\,$};
	\node (4) [white vertex] at (0) {};
	
	\node (T0) [color=green!75!blue] at ($(1) + (1,0.05)$) {\(T_0\)};
	\node (T1) [color=red] at ($(3) + (1,0.05)$) {\(T_1\)};
	
	\draw[M edge, color=green!75!blue] (0) -- (1);
	\draw[F1 edge, color=green!75!blue] (1) -- (2);
	\draw[F1 edge, color=green!75!blue] (2) -- (3);
	\draw[F2 edge, color=green!75!blue] (3) -- (4);
	
	\node (00) [white vertex] at (0) {};
	\node (11) [black vertex] at (4,0) {};
	\node (a11) [left] at (11) {\small$a^1_{1}\,\,$};
	\node (22) [white vertex] at (3,-1) {};
	\node (a22) [below] at (22) {\small$a^1_{2}$};
	\node (33) [black vertex] at (2,0) {};
	\node (a33) [right] at (33) {\small$\,a^1_{3}$};
	\node (44) [white vertex] at (3,1) {};
	\node (a44) [left] at (44) {\small$a^1_{4}\,\,\,$};

	\node (T2) [color=blue] at ($(11) + (1,0.05)$) {\(T_2\)};

	\draw[M edge, color=red] (00) to[bend left=55] (11);
	\draw[F1 edge, color=red] (11) -- (22);
	\draw[F1 edge, color=red] (22) -- (33);
	\draw[F2 edge, color=red] (33) -- (44);

	\node (000) [white vertex] at (5,1) {};
	\node (a000) [left] at (000) {\small$a^2_{0}\,\,$};
	\node (111) [black vertex] at (4,0) {};
	\node (a111) [right] at (111) {\small$\,\,\,a^2_{1}$};
	\node (222) [white vertex] at (5,-1) {};
	\node (a222) [below] at (222) {\small$a^2_{2}$};
	\node (333) [black vertex] at (6,0) {};
	\node (a333) [left] at (333) {\small$a^2_{3}\,\,$};
	\node (444) [white vertex] at (0) {};
	
	\node (T3) [color=black] at ($(333) + (1,0.05)$) {\(T_3\)};
	
	\draw[M edge, color=blue] (000) -- (111);
	\draw[F1 edge, color=blue] (111) -- (222);
	\draw[F1 edge, color=blue] (222) -- (333);
	\draw[F2 edge, color=blue] (333) to[bend right=60] (444);

	\node (0000) [white vertex] at (7,1) {};
	\node (a0000) [right] at (0000) {\small$\,\,a^3_{0},a^3_{4}$};
	\node (1111) [black vertex] at (8,0) {};
	\node (a1111) [right] at (1111) {\small$\,a^3_{1}$};
	\node (2222) [white vertex] at (7,-1) {};
	\node (a2222) [below] at (2222) {\small$a^3_{2}$};
	\node (3333) [black vertex] at (6,0) {};
	\node (a3333) [right] at (3333) {\small$\,a^3_{3}$};
	\node (4444) [white vertex] at (0000) {};

	\draw[M edge] (0000) -- (1111);
	\draw[F1 edge] (1111) -- (2222);
	\draw[F1 edge] (2222) -- (3333);
	\draw[F2 edge] (3333) -- (4444);
\end{tikzpicture}
		\begin{tikzpicture}[scale = 1.2]

	
	\node (p1) at (1,1) {};
	\node (q1) at (3,1) {};
	\node (q2) at (5,1) {};
	\node (q3) at (7,1) {};

	\node (0) [white vertex] at (1,1) {};
	\node (a0) [left] at (0) {\small$a^0_{0},a^0_{4},a^1_{0},a^2_{4}\,\,\,$};
	\node (1) [black vertex] at (0,0) {};
	\node (a1) [left] at (1) {\small$a^0_{1}\,$};
	\node (2) [white vertex] at (1,-1) {};
	\node (a2) [below] at (2) {\small$a^0_{2}$};
	\node (3) [black vertex] at (2,0) {};
	\node (a3) [left] at (3) {\small$a^0_{3}\,\,$};
	\node (4) [white vertex] at (0) {};
	
	\node (T0) [color=green!75!blue] at ($(1) + (1,0.05)$) {\(T_0'\)};
	\node (T1) [color=red] at ($(3) + (1,0.05)$) {\(T_1'\)};
	
	\draw[M edge, color=green!75!blue] (0) -- (1);
	\draw[F1 edge, color=green!75!blue] (1) -- (2);
	\draw[F1 edge, color=green!75!blue] (2) -- (3);
	\draw[F2 edge, color=red] (3) -- (4);
	
	\node (00) [white vertex] at (0) {};
	\node (11) [black vertex] at (4,0) {};
	\node (a11) [left] at (11) {\small$a^1_{1}\,\,$};
	\node (22) [white vertex] at (3,-1) {};
	\node (a22) [below] at (22) {\small$a^1_{2}$};
	\node (33) [black vertex] at (2,0) {};
	\node (a33) [right] at (33) {\small$\,a^1_{3}$};
	\node (44) [white vertex] at (3,1) {};
	\node (a44) [left] at (44) {\small$a^1_{4}\,\,\,$};

	\node (T2) [color=blue] at ($(11) + (1,0.05)$) {\(T_2'\)};
	
	\draw[M edge, color=blue] (00) to[bend left=55] (11);
	\draw[F1 edge, color=red] (11) -- (22);
	\draw[F1 edge, color=red] (22) -- (33);
	\draw[F2 edge, color=green!75!blue] (33) -- (44);

	\node (000) [white vertex] at (5,1) {};
	\node (a000) [left] at (000) {\small$a^2_{0}\,\,$};
	\node (111) [black vertex] at (4,0) {};
	\node (a111) [right] at (111) {\small$\,\,\,a^2_{1}$};
	\node (222) [white vertex] at (5,-1) {};
	\node (a222) [below] at (222) {\small$a^2_{2}$};
	\node (333) [black vertex] at (6,0) {};
	\node (a333) [left] at (333) {\small$a^2_{3}\,\,$};
	\node (444) [white vertex] at (0) {};
	
	\node (T3) [color=black] at ($(333) + (1,0.05)$) {\(T_3'\)};
	
	\draw[M edge, color=red] (000) -- (111);
	\draw[F1 edge, color=blue] (111) -- (222);
	\draw[F1 edge, color=blue] (222) -- (333);
	\draw[F2 edge, color=black] (333) to[bend right=60] (444);

	\node (0000) [white vertex] at (7,1) {};
	\node (a0000) [right] at (0000) {\small$\,\,a^3_{0},a^3_{4}$};
	\node (1111) [black vertex] at (8,0) {};
	\node (a1111) [right] at (1111) {\small$\,a^3_{1}$};
	\node (2222) [white vertex] at (7,-1) {};
	\node (a2222) [below] at (2222) {\small$a^3_{2}$};
	\node (3333) [black vertex] at (6,0) {};
	\node (a3333) [right] at (3333) {\small$\,a^3_{3}$};
	\node (4444) [white vertex] at (0000) {};

	\draw[M edge] (0000) -- (1111);
	\draw[F1 edge] (1111) -- (2222);
	\draw[F1 edge] (2222) -- (3333);
	\draw[F2 edge, color=blue] (3333) -- (4444);
\end{tikzpicture}
		\caption{Example of a sequence \(T_0,T_1,T_2,T_3\) and the corresponding 
paths \(T'_0\), \(T'_1\), \(T'_2\), \(T'_3\).}
		\label{fig:disentangle}
	\end{figure}
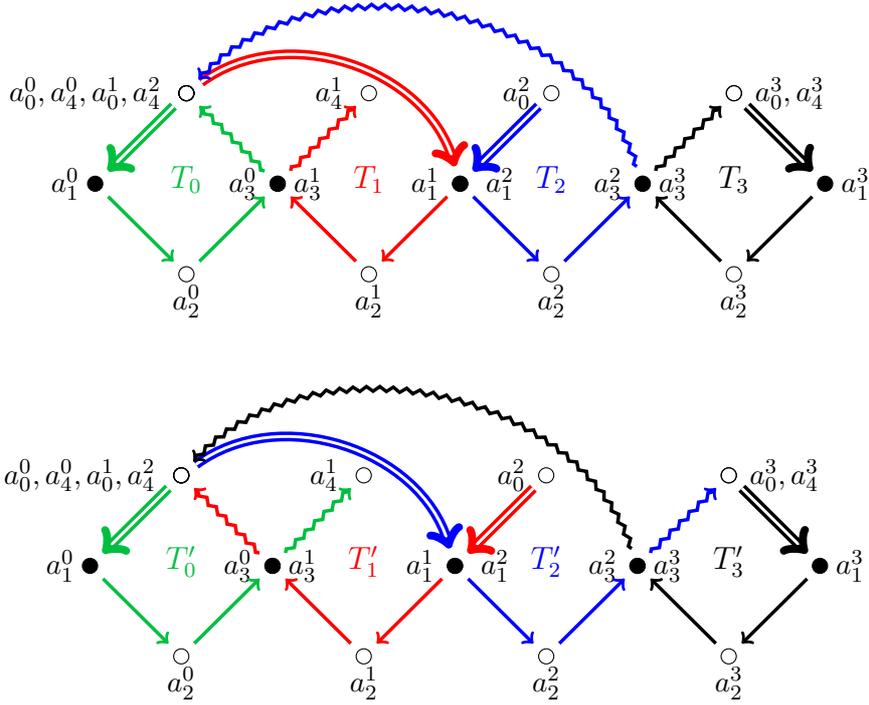	
	Furthermore, \(\rho (\mathcal{D}') < \rho (\mathcal{D})\), contradicting the minimality of 
\(\rho(\mathcal{D})\). Therefore, \(G'\) admits an \(\mathcal{F}'\)-basic path decomposition $\mathcal{D}$.

To finish the proof we extend the \(\mathcal{F}'\)-basic {path} decomposition $\mathcal{D}$ of $G'$ to an 
\hbox{\(\mathcal{F}\)-canonical} decomposition of $G$ by using the edges of $H^-(A)$. 
Note that each {\(\mathcal{F}\)-basic path}  in~\(\mathcal{D}\)
is a directed path ending {with} an edge of \(F^+_2(A)\)
and at a vertex of \(B\). But since, by item~(iii) of Definition~\ref{def:special-decomposition}, 
$d^-_{H}(v)=d^+_{H}(v)$ for every $v\in B$, we can easily extend~$\mathcal{D}$ to an 
\(\mathcal{F}\)-canonical decomposition of $G$ by adding one edge of $H^-(A)$ 
to each one of {its} \(\mathcal{F}'\)-basic paths,
{concluding} the proof.
\end{proof}

Combining Lemmas~\ref{lemma:special-decomposition}~and~\ref{lemma:decomposition-Tout} we obtain the following 
corollary.

\begin{corollary}\label{cor:special-decomposition}
Let $G=(A \cup B,E)$ be a $6$-edge-connected bipartite graph such that the vertices in $A$ have degree divisible by $5$. 
Then, $G$ is the underlying graph of a directed graph $\vec G$ that
admits a fractional factorization \(\mathcal{F}\) 
and an \(\mathcal{F}\)-canonical decomposition.
\end{corollary}

\section{Proof of the Main Theorem}\label{sec:main}

In this section we manage to 
``disentangle'' the trails of a canonical decomposition to obtain a decomposition into paths of length~$5$.
Denote by \(T_5\) the only bipartite trail of length \(5\) that is not a path.
We recall that a \(\{P_5,T_5\}\)-decomposition \(\mathcal{D}\) of a directed graph \(\vec G\) is a 
decomposition of \(\vec G\) such that every element of \(\mathcal{D}\) is either a copy of \(P_5\)
or a copy of \(T_5\).

Let \(\vec G\)
be a directed graph and { \(ab\)
 an edge of \(\vec G\).}
Let \(\mathcal{D}\)
be a decomposition of \(\vec G\),
and let \(T\)
be the element of \(\mathcal{D}\)
that contains \(ab\).
We say that \(ab\)
is \emph{inward} in \(\mathcal{D}\)
if \(d_T(a) = 1\).
Suppose that $\vec G$ admits a fractional factorization
$\mathcal{F} = (M,F,H)$.  Let \(\mathcal{D}\)
be a $\{P_{5},T_{5}\}$-decomposition of $\vec G$.  We say that
\(\mathcal{D}\)
is \emph{\(M\)-complete}
if every edge of \(M\)
is inward in \(\mathcal{D}\).
Note that if \(T\)
is an \(\mathcal{F}\)-canonical
path or an \(\mathcal{F}\)-canonical
trail, then the edge of \(M\)
in \(T\)
is inward in \(\mathcal{D}\).
Therefore, if \(\mathcal{D}\)
is an \(\mathcal{F}\)-canonical
decomposition, then \(\mathcal{D}\)
is \(M\)-complete.  The next theorem is our main result.

\begin{theorem}\label{theorem:main-theorem}
  There exists a natural number $k_T$ such that, if \(G\) is a
  \(k_T\)-edge-connected graph and \(|E(G)|\) is divisible by \(5\),
  then \(G\) admits a \(P_5\)-decomposition.
\end{theorem}

Our main theorem follows directly from Theorem~\ref{theorem:conj-equivalence} and the next 
result.

\begin{theorem}\label{theorem:main-theorem-bip}
If \(G\) is a \(48\)-edge-connected bipartite graph and \(|E(G)|\) is
divisible by \(5\), then \(G\) admits a \(P_5\)-decomposition.
\end{theorem}

\begin{proof}

Let $G = (A \cup B,E)$ be a \(48\)-edge-connected bipartite graph such that \(|E|\) is
divisible by~\(5\). 
By Lemma~\ref{lem:tight-good-initial-decomposition} (taking $r=6$ and $k=5$), 
\(G\) can be decomposed into graphs \(G_1\) and \(G_2\) 
such that \(G_1\) is \(6\)-edge-connected and \(d_{G_1} (v)\) is divisible
by \(5\) for every \(v \in A\),
and \(G_2\) is \(6\)-edge-connected and \(d_{G_2} (v)\) is divisible
by \(5\) for every \(v \in B\). 
Thus, by Corollary~\ref{cor:special-decomposition}, $G_i$ is 
the underlying graph of a directed graph $\vec {G_i}$ that admits a fractional factorization 
$\mathcal{F}_i =(M_i,F_i,H_i)$ and an
\(\mathcal{F}_i\)-canonical decomposition $\mathcal{D}_i$, for \(i=1,2\).

By definition, \(\mathcal{D}_1\) is an \(M_1\)-complete decomposition of \(G_1\) 
and \(\mathcal{D}_2\) is an \(M_2\)-complete decomposition of \(G_2\). 
Let $M=M_1\cup M_2$ and \(\mathcal{F} = (M,F_1\cup F_2,H_1\cup H_2)\).
Then, \(\mathcal{D} = \mathcal{D}_1 \cup \mathcal{D}_2\) 
is an \(M\)-complete \(\mathcal{F}\)-canonical decomposition 
of $\vec G$, where $\vec G=\vec{G_1}\cup \vec{G_2}$. Note that, for every vertex $v$ of $\vec G$, there is at least one edge of $M$ pointing to $v$.
Moreover, since an \(\mathcal{F}\)-canonical path is a copy of \(P_5\), 
and an \(\mathcal{F}\)-canonical trail is a copy of \(T_5\), 
we have that any \(\mathcal{F}\)-canonical decomposition of \(\vec G\) is
also a \(\{P_5,T_5\}\)-decomposition of \(\vec G\).
Therefore, \(\mathcal{D}\) is an \(M\)-complete \(\{P_5,T_5\}\)-decomposition of \(\vec G\).

Let \(\mathcal{D}\) be an $M$-complete \(\{P_5,T_5\}\)-decomposition of \(\vec G\) with minimum number of 
copies 
of $T_5$. If there is no copy of $T_5$ in $\mathcal{D}$, then $\mathcal{D}$ is a $P_5$-decomposition 
of 
$\vec G$ and the proof is complete. Therefore, we may suppose that there is at least one copy of $T_5$ 
in~$\mathcal{D}$. In what follows, we aim for a contradiction.

Let $T=v_0v_1v_2v_3v_4v_5$ with $v_5=v_1$ be a copy of $T_5$ in $\mathcal{D}$. 
Recall that there exists an edge \(uv_2\) of $M$ pointing to $v_2$. Let 
$B_1$ be the element of $\mathcal{D}$ that contains \(uv_2\).
Since \(\mathcal{D}\) is \(M\)-complete, \(d_{B_1}(u) = 1\).
Therefore, we may suppose that \(B_1=b_0b_1b_2b_3b_4b_5\), where $b_1=v_2$,
and, possibly, $b_1=b_5$.

We divide the proof in two cases, depending on whether \(v_1\)
belongs or not to \(V(B_1)\).

\noindent{\textbf{Case 1}}: \(v_1\notin V(B_1)\).

\noindent
	Let \(T' = v_0v_1v_4v_3v_2b_0\), \(B_1' = v_1b_1b_2b_3b_4b_5\),
	and \(\mathcal{D}' = \mathcal{D} - T - B_1 + T' + B_1'\).
	We claim that \(T'\) is a path, 
	\(B_1'\) is of the same type of element as $B_1$ 
	(i.e., the underlying graphs of \(B_1'\) and \(B_1\) are isomorphic), 
	and the edges of \(M\) in 
	\(A(T')\cup A(B_1')\) are inward in \(\mathcal{D}'\).
	Thus \(\mathcal{D}'\) is an $M$-complete
	decomposition with fewer copies of \(T_5\) than~$\mathcal{D}$, a contradiction.
	
	First, let us prove that \(T'\) is a path. Note that \(b_0 \neq v_0\) and \(b_0 \neq v_4\),
	otherwise \(b_0b_1v_1\) would induce a triangle in \(G\), a contradiction.
	We also know that \(b_0 \neq v_1\) and \(b_0 \neq v_3\), since \(G\) has no parallel edges. Furthermore,
	\(b_0 \neq v_2\), since \(G\) has no loops.
	Since \(v_1 \notin V(B_1)\), if \(B_1\) is a path, then \(B_1'\) is a path;
	and \(B_1'\) is a copy of \(T_5\), otherwise.
	
	It is left to prove that every directed edge in \(M\) is inward in $\mathcal{D}'$.
	We just need to prove this for the directed edges in \(M\cap \big(A(T')\cup A(B_1')\big)\).
	Note that the only edges in \(M\cap \big(A(T')\cup A(B_1')\big)\) are \(b_0v_2\) and, 
	possibly, \(v_0v_1\) and \(b_5b_4\). 
	Since \(d_{T'}(b_0) = 1\) and \(d_{T'}(v_0) = 1\), 
	the edges \(b_0b_1\) and \(v_0v_1\) are inward in 
	\(\mathcal{D}'\).
	If \(b_5b_4\) is an edge of \(M\), then \(B_1\) is a path ending at \(b_5\).
	Therefore, \(B_1'\) is a path ending at \(b_5\),
	and \(b_5b_4\) is inward in \(\mathcal{D}'\). \ \\

\noindent{\textbf{Case 2}}: \(v_1\in V(B_1)\).

\noindent
	Consider a sequence $\mathcal{B}=B_1B_2\ldots B_{k-1}$ of elements of $\mathcal{D}$, 
	where $b^1_1 = v_2$, $B_i=b^i_0b^i_1b^i_2b^i_3b^i_4b^i_5$ for $i\leq k-1$.
	We say that \(\mathcal{B}\) is a \emph{coupled sequence centered at \(v_1\)} 
	if the following properties hold
	(See Figure~\ref{fig:main-case2}).
	\begin{enumerate}[label=(\roman*)]
		\item $b^{i}_0b^{i}_1\in M$, for $1\leq i\leq k-1$; 
		\item $b^{i}_1=b^{i-1}_3$, for $2\leq i\leq k-1$; 
		\item $b^{i}_4=v_1$, for $1\leq i\leq k-1$. 
	\end{enumerate}	
	Note that, by hypothesis, \(v_1\) is a vertex of \(B_1\).
	Since \(G\) is a bipartite graph, 
	\(v_1 = b^1_4\).
	Therefore, \(B_1\) is a coupled sequence centered at \(v_1\) 
	with only one element (that is, \(k=2\)).
	Thus, we may suppose that there is a maximal coupled sequence $\mathcal{B}$
	centered at \(v_1\).

\begin{claim}
	\(B_i\) is a path of length \(5\), for \(1\leq i\leq k-1\).
\end{claim}

\begin{proof}
If for some \(i\in \{1,\ldots, k-1\}\), the element \(B_i\) is a copy of \(T_5\),
then \(d_{B_i}(b^i_0) = 1\) and \(b^i_5 = b^i_1\), 
because (by item (i)) \(b^i_0b^i_1\) is an edge of \(M\) and, 
since \(\mathcal{D}\) is \(M\)-complete, \(b^i_0b^i_1\) must be inward in \(\mathcal{D}\).
Since \(v_1 \in V(B_i)\), we know that either \(v_1 = b^i_2\) or \(v_1=b^i_4\),
because \(G\) is bipartite.
Note that the edge \(v_2v_1\) is an edge of \(T\).
If \(i = 1\), then \(b^1_1v_1\) and \(v_2v_1\) are parallel edges.
If \(i > 1\), then (by item (ii)) \(b^{i-1}_3v_1 = b^i_1v_1\) must be an edge of \(B_{i-1}\)
and of \(B_i\), and \(\mathcal{D}\) covers this edge twice.
Therefore, for every $1\leq i\leq k-1$, the element $B_i$ is a copy of $P_5$.
\end{proof}
\begin{figure}[h]
	\centering
	\scalebox{.85}{\begin{tikzpicture}[scale = 1.5]

\node (c) [] at (0,0) {};

\foreach \i in {0,1,2,3,5,6,7}
    {
	\node (o\i) [] at (\i*45-45:3) {};
    };
    
\foreach \i in {1,2}
    {
	\node (oo\i) [] at (\i*45-45:4) {};
    };

\foreach \i in {0,1,2}
    {
	\node (m\i) [] at (\i*45-22.5:2) {};
	\node (lm\i) [] at (\i*45-22.5:2.25) {};
    };
    
\foreach \i in {0,1,2}
    {
	\node (l\i) [] at (\i*45-22.5:1.5) {};
    };

	\node (x0) [black vertex] at (o7) {};
	\node (v0) [left] at ($(x0) - (0.05,0)$) {\(v_0\)};
	\node (x1) [white vertex] at (c) {};
	\node (v1) [below] at ($(x1) + (-1.1,0.5)$) {\(b_4^2 = b_4^1 = v_1 = v_5\)};
	\node (x2) [black vertex] at (o1) {};
	\node (v2) [below] at ($(x2) + (0.23,-0.08)$) {\(v_2 = b_1^1\)};
	\node (x3) [white vertex] at (m0) {};
	\node (v3) [below] at (lm0) {\(v_3\)};
	\node (x4) [black vertex] at (o0) {};
	\node (v4) [below] at ($(x4) + (0.23,-0.08)$) {\(v_4\)};
	
	\draw[M edge][color=green!75!blue] (x0) -- (x1);
	\draw[E edge][color=green!75!blue] (x1) -- (x2);
	\draw[E edge][color=green!75!blue] (x2) -- (x3);
	\draw[E edge][color=green!75!blue] (x3) -- (x4);
	\draw[E edge][color=green!75!blue] (x4) -- (x1);

	\node (T) [color=green!75!blue] at (l0) {\(T\)};
	\node (B1) [color=red] at (l1) {\(B_1\)};
	\node (B2) [color=blue] at (l2) {\(B_2\)};


	\node (y0) [white vertex] at (oo1) {};
	\node (b01) at ($(y0) + (0.2,-0.2)$) {\(b_0^1\)};
	\node (y1) [black vertex] at (x2) {};
	\node (y2) [white vertex] at (m1) {};
	\node (b21) at (lm1) {\(b_2^1\)};
	\node (y3) [black vertex] at (o2) {};
 	\node (b31) [above] at ($(y3) + (0.6,-0.4)$) {\(b_3^1 = b_1^2\)};
	\node (y4) [white vertex] at (c) {};
	\node (y5) [black vertex] at (o6) {};
	\node (b51) [left] at ($(y5) - (0.05,0)$) {\(b_5^1\)};

	\draw[M edge][color=red] (y0) -- (y1);
	\draw[E edge][color=red] (y1) -- (y2);
	\draw[E edge][color=red] (y2) -- (y3);
	\draw[E edge][color=red] (y3) -- (y4);
	\draw[E edge][color=red] (y4) -- (y5);

	\node (z0) [white vertex] at (oo2) {};
	\node (b02) at ($(z0) + (0.2,-0.2)$) {\(b_0^2\)};
	\node (z1) [black vertex] at (o2) {};
	\node (z2) [white vertex] at (m2) {};
	\node (b32) [right] at ($(lm2)-(0.2,0)$) {\(b_2^2\)};
	\node (z3) [black vertex] at (o3) {};
	\node (b32) [right] at ($(z3) + (0.05,0.2)$) {\(b_3^2\)};
	\node (z4) [white vertex] at (c) {};
	\node (z5) [black vertex] at (o5) {};
	\node (b52) [left] at ($(z5) - (0.05,0)$) {\(b_5^2\)};
	
	\draw[M edge][color=blue] (z0) -- (z1);
	\draw[E edge][color=blue] (z1) -- (z2);
	\draw[E edge][color=blue] (z2) -- (z3);
	\draw[E edge][color=blue] (z3) -- (z4);
	\draw[E edge][color=blue] (z4) -- (z5);

\end{tikzpicture}}
	\caption{Example of a trail $T=v_0v_1v_2v_3v_4v_5$ with
          $v_5=v_1$, and a coupled sequence \(B_1,B_2\) centered at $v_1$.}
	\label{fig:main-case2}
\end{figure}
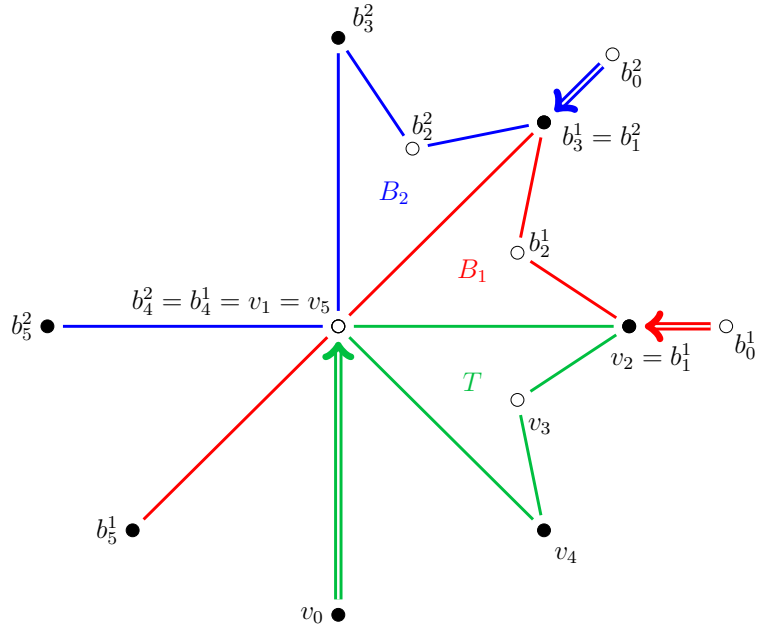	

\begin{claim}
\(B_i \neq B_j\), for \(1\leq i<j\leq k-1\).
\end{claim}
\begin{proof}
  Suppose, by contradiction, that $\mathcal{B}$ has repeated
  elements.  Let $B_i$ be the first element of $\mathcal{B}$ such that
  $B_i=B_j$ for some $j$ with $i<j$.  Since $b^{i}_0b^{i}_1\in M$ and
  $b^{j}_0b^{j}_1\in M$ (item (i)), and the elements of $B$ belong to
  an $M$-complete decomposition, either $b^j_0=b^i_0$ or
  $b^j_0=b^i_5$. If $b^j_0=b^i_5$, then we know that $b^j_4=b^i_1=v_1$
  (by item (iii)), from where we conclude that \(B_i\) contains the
  triangle \(b_4^j\,b_3^j\,b_2^j\,b_4^j\), a contradiction.  Therefore,
  assume that $b^j_0=b^i_0$. Note that \(b^{j-1}_3 =b^j_1 = b^i_1\)
  (by item (ii)). Also, \(i > 1\), otherwise \(b^{j-1}_3 = v_2\) and
  \(b^{j-1}_3b^{j-1}_4 = v_2v_1 \in E(B_{j-1})\), but \(v_1v_2 \in
  E(T)\) and $T$ and $B_{j-1}$ are different, by the choice of \(i\).
  Therefore, by item (iii), \(b^{j-1}_4 = b^{i-1}_4 = v_1\), implying
  that \(b^{i-1}_3b^{i-1}_4 = b^{j-1}_3b^{j-1}_4\) and, then,
  \(B_{i-1} = B_{j-1}\), a contradiction to the minimality of
    \(i\).  Therefore, \(B_i \neq B_j\) for every \(1\leq i<j\leq
  k-1\).
\end{proof}


Recall that there is at least one edge \(e\) in \(M\) pointing to \(b^{k-1}_3\).
Let $B_k$ be the element of $\mathcal{D}$ that contains \(e\).
We may suppose that $B_k=b^k_0\,b^k_1\,b^k_2\,b^k_3\,b^k_4\,b^k_5$, where \(e = b^k_0b^k_1\).
Note that \(\mathcal{B}' = B_1B_2\cdots B_{k-1}B_k\) satisfies items (i) and (ii).
Also, item (iii) holds for \(1\leq i\leq k-1\).
Since \(\mathcal{B}\) is maximal, \(\mathcal{B}'\) is not a coupled sequence.
Thus, item (iii) does not hold for \(i = k\).
Therefore, \(b^k_4 \neq v_1\).

Now consider the following elements:
\begin{itemize}
	\item	\(T' = T - v_2v_1 + b_0^{{ 1}}b_1^{{ 1}}\).
	\item	\(B_1' = B_1 - b_0^{{ 1}}b_1^{{ 1}} +v_2v_1 - b_3^{{ 1}}v_1 + b_0^2b_1^2\).
	\item	\(B_i' = B_i - b_0^ib_1^i + b_3^{i-1}v_1 - b_3^iv_1 + b_0^{i+1}b_1^{i+1}\), for \(2\leq i \leq k-1\).
	\item	\(B_k' = B_k - b_0^kb_1^k + b_3^{k-1}v_1\).
\end{itemize}

We claim that \(T',B_1',\ldots,B_{k-1}'\) are paths and 
$B_k'$ is of the same type of element as \(B_k\). 
The following arguments are very similar to the ones above, 
we present them for completeness.

To check that \(T'\) is a path, we prove that \(b^1_0\notin V(T)-v_0\). 
Note that \(b^1_0 \neq v_0\) and \(b^1_0 \neq v_4\),
otherwise \(b^1_0b^1_1v_1\) would induce a triangle in \(G\). 
Also \(b^1_0 \neq v_1\) and \(b^1_0 \neq v_3\), 
because \(G\) has no parallel edges, and since $G$ has no loops, \(b^1_0 \neq v_2\). 
Therefore,  \(T'\) is a path.
	
Let us check that \(B_i'\) is a path for \(1\leq i\leq k-1\).
Since \(V(B_i') = V(B_i) - b^i_0 + b^{i+1}_0\), 
we just have to prove that \(b^{i+1}_0 \notin\{b^i_1,b^i_2,b^i_3,b^i_4,b^i_5\}\).
If \(b^{i+1}_0 = b^i_1\), then \(b^i_1\,b^i_2\,b^i_3\,b^{i+1}_0\) is a triangle in \(G\).
If \(b^{i+1}_0 = b^i_2\), then \(b^i_3b^i_2\) and \(b^{i+1}_1b^{i+1}_0\) are parallel edges.
Since \(b^i_3 = b^{i+1}_1\) and \(b^{i+1}_1\neq b^{i+1}_0\), we have \(b^{i+1}_0\neq b^i_3\).
If \(b^{i+1}_0 = b^i_4\), then \(b^i_3b^i_4\) and \(b^{i+1}_1b^i_0\) are parallel.
If \(b^{i+1}_0 = b^i_5\), then \(b^{i+1}_0b^i_3\,b^i_4\,b^i_5\) is a triangle in \(G\).
Therefore, $B_2',\ldots,B_{k-1}'$ are paths.

Now, let us prove that \(v_1 \notin\{b_1^k,b_2^k,b_3^k,b_4^k,b_5^k\}\).
Since $b^k_1=b^{k-1}_3$ and $G$ is bipartite, we conclude that $v_1\notin\{b^k_1,b^k_3,b^k_5\}$.
Furthermore, since $b^k_1=b^{k-1}_3$ and $b^{k-1}_3v_1\in E(G)$, we conclude that $b^k_2\neq v_1$. 
By the maximality of the  sequence $\mathcal{B}$, we conclude that $b^k_4\neq v_1$. 
Thus, \(B_k'\) is a trail.
If \(b^k_5 \neq b^k_1\), then \(B_k\) and \(B_k'\) are both paths of length five.
If \(b^k_5 = b^k_1\), then \(B_k\) and \(B_k'\) are both copies of \(T_5\).
Therefore, \(B_k'\) is of the same type of element as~\(B_k\).

Let \(\mathcal{D}' = \mathcal{D} - T - B_1 - \cdots - B_k + T' + B_1' + \cdots + B_k'\).	
Since the edges of \(M\) are \(b_0^ib_1^i\) and, possibly \(b_5^ib_4^i\),
every edge of \(M\) is inward in \(\mathcal{D}'\).
Therefore, \(\mathcal{D}'\)
is an \(M\)-complete decomposition with fewer copies of \(T_5\) than $\mathcal{D}$, a 
contradiction.
\end{proof}

\section{Concluding remarks}\label{sec:concluding}

The technique we have shown here (in Section~\ref{sec:main}) to
  disentangle elements of the canonical decomposition seems to be
  useful to deal with more general structures.  Besides our current
  work~\cite{BoMoOsWa15+thomassen} on generalizations of these results
  to show that Conjecture~\ref{conj:dec} holds for paths of any fixed
  length, in another direction, we were able to prove a variant of our
  results to deal with \(P_\ell\)-decompositions of regular graphs of
  prescribed girth~\cite{BoMoOsWa15+reg}.  These results were obtained
  by combining ideas from this paper and a special result, which we
  named ``Disentangling Lemma'', that generalizes the ideas used in
  Section~\ref{sec:main}.  We were not able to generalize
  Lemma~\ref{lemma:decomposition-Tout} and
  Corollary~\ref{cor:special-decomposition} to obtain decompositions
  into paths of any given length. But, considering more 
    powerful factorizations and higher connectivity, we can obtain
   a kind of generalized versions of these results.
\bibliographystyle{amsplain_yk}
\bibliography{bibliografia}
\end{document}